\documentclass{amsart}

\usepackage{amssymb}
\usepackage{graphicx}
\usepackage{MnSymbol}

\usepackage{amsmath,amscd}

\usepackage{amsthm}

\title[might never be free]{Deeply concatenable subgroups might never be free}

\theoremstyle{definition}\newtheorem{theorem}{Theorem}
\theoremstyle{definition}
\theoremstyle{definition}\newtheorem*{B}{Theorem \ref{lengthfunctiontheorem}}

\theoremstyle{definition}\newtheorem{bigtheorem}{Theorem}

\numberwithin{theorem}{section}
\theoremstyle{definition}\newtheorem{corollary}[theorem]{Corollary}
\theoremstyle{definition}\newtheorem{proposition}[theorem]{Proposition}
\theoremstyle{definition}\newtheorem{definition}[theorem]{Definition}
\theoremstyle{definition}
\theoremstyle{definition}\newtheorem{example}[theorem]{Example}
\theoremstyle{definition}
\theoremstyle{definition}
\theoremstyle{definition}\newtheorem{lemma}[theorem]{Lemma}
\theoremstyle{definition}
\theoremstyle{definition}
\theoremstyle{definition}
\theoremstyle{definition}

\newcommand{\W}{\mathcal{W}}
\newcommand{\HEG}{\operatorname{HEG}}
\newcommand{\Bd}{\mathcal{BD}}
\newcommand{\Scat}{\mathcal{SCAT}}
\newcommand{\proj}{\operatorname{proj}}
\newcommand{\Red}{\operatorname{Red}}

\begin{document}

\author[Samuel M. Corson]{Samuel M. Corson}
\address{Ikerbasque- Basque Foundation for Science and Matematika Saila, UPV/EHU, Sarriena S/N, 48940, Leioa - Bizkaia, Spain}
\email{sammyc973@gmail.com}

\author[Saharon Shelah]{Saharon Shelah}
\address{Einstein Institute of Mathematics, The Hebrew University of Jerusalem, Jerusalem 91904 Israel}
\address{Department of Mathematics, Rutgers University, Piscataway, NJ 08854 USA}
\email{shelah@math.huji.ac.il}

\thanks{The research of the first author was
supported by European Research Council grant PCG-336983.  The research of the second author was partially supported by European Research Council grant 338821.  Paper 1142 on Shelah's list.}

\begin{abstract}  We show that certain algebraic structures lack freeness in the absence of the axiom of choice.  These include some subgroups of the Baer-Specker group $\mathbb{Z}^{\omega}$ and the Hawaiian earring group.  Applications to slenderness, completely metrizable topological groups, length functions and strongly bounded groups are also presented.
\end{abstract}

\maketitle

\begin{section}{Introduction}

The axiom of choice has been an essential tool in understanding uncountable groups.  It is used in producing nonconstructive homomorphisms, selecting Hamel bases, and determining the freeness of certain groups.  We explore some algebraic consequences that  hold under an alternative to this axiom.

Let ZF denote the Zermelo-Fraenkel axioms of set theory and ZFC denote ZF plus the axiom of choice.  For sets $X$ and $Y$ we write $|X| \leq |Y|$ if there is an injection from $X$ to $Y$.  The axiom of choice is equivalent to the assertion that for any sets $X$ and $Y$ at least one of $|X| \leq |Y|$ or $|Y|\leq |X|$ holds.  Recall that the \emph{axiom of dependent choice}, denoted DC, is the statement that if $R$ is a binary relation on a set $X \neq \emptyset$ such that $(\forall x)(\exists y)(xRy)$ there exists a sequence $\{z_n\}_{n\in\omega}$ such that $z_n R z_{n+1}$ for all $n$.  Consequences of DC include the axiom of countable choice as well as the assertion that for any set $X$ at least one of $|\omega|\leq |X|$ or $|X| \leq |\omega|$ holds.  Clearly DC is a consequence of the axiom of choice.

We will utilize DC in conjunction with an assertion that the abstract set theory of the Cantor set $2^{\omega}$ is inextricably tied to topology.  The precise formulation of this assertion is that all subsets of $2^{\omega}$ have the \emph{Baire property} (see Section \ref{TheoremA}).  Let $(*)$ denote the theory ZF + DC + ``Every subset of $2^{\omega}$ has the Baire property."  In his celebrated work \cite{So} Solovay showed that $(*)$ is consistent provided one assumes the existence of an inaccessible cardinal.  The second author showed in \cite{Sh1} that ZF is equiconsistent with $(*)$.  Almost all consequences of $(*)$ in this note are known to be false under ZFC and so all such consequences are independent of ZF + DC.

In Section \ref{TheoremA} we establish the first of the main results.  The notions $g_n\searrow 1_G$ and  \emph{deeply concatenable} are defined in Section \ref{TheoremA} immediately before the proof of Theorem \ref{thebigone}.  They essentially state that any subsequence may be concatenated in such a way that the product is a well defined limit in the group (Definitions \ref{Hausdorffarrow} and \ref{Looparrow}).  

\begin{bigtheorem}\label{thebigone}  $(*)$ The following hold:

\begin{enumerate}  
\item  Suppose $G$ is a Hausdorff topological group, $g_n\searrow 1_G$ and $\phi: G \rightarrow H$ is a homomorphism.  If the $g_n$ are not eventually in $\ker(\phi)$ then $|2^{\omega}|\leq|H|$.

 \item  Suppose $(l_n)_{n\in \omega}$ is deeply concatenable for a subgroup $K\leq \pi_1(X, x)$ and $\phi: K \rightarrow H$ is a homomorphism.  If the $[l_n]$ are not eventually in $\ker(\phi)$ then $|2^{\omega}|\leq|H|$.
\end{enumerate}
\end{bigtheorem}

\noindent One immediate consequence of Theorem \ref{thebigone} is a strong affirmative answer to a question of Cannon and Conner regarding fundamental groups (Corollary \ref{CannonConner}). 

In Section \ref{FromA} further consequences of Theorem \ref{thebigone} are presented.  We recover the result of Blass \cite{Bl} that the group of bounded sequences in $\mathbb{Z}^{\omega}$, classically known to be free abelian under ZFC \cite{No}, is not free under $(*)$.  Additionally, two subgroups of the Hawaiian earring group fail to be free under $(*)$.  These are the so-called non-abelian Specker group \cite{Z} and the subgroup of scattered words \cite{CC}, both of which are free assuming choice.

Theorem \ref{thebigone} also has applications to the theory of slender groups.  A consequence of $(*)$ is that every nonslender abelian group $A$ satisfies $|2^{\omega}| \leq |A|$ (Theorem \ref{slender}).  Contrast this with the situation under the axiom of choice: any group including $\mathbb{Q}$ or torsion fails to be slender.  The comparable statement and contrast hold for the noncommutatively slender groups. 

In Section \ref{Length} we prove a result which is similar in flavor to Theorem \ref{thebigone}:
\begin{bigtheorem}\label{lengthfunctiontheorem}  $(*)$ The following hold:

\begin{enumerate}  
\item  Suppose $G$ is a Hausdorff topological group, $g_n\searrow 1_G$ and $L$ is a length function on $G$ (see Definition \ref{length}).  The sequence $(L(g_n))_{n\in \omega}$ is bounded.

 \item  Suppose $(l_n)_{n\in \omega}$ is deeply concatenable for $K\leq \pi_1(X, x)$ and $L$ is a length function on $K$.  The sequence $(L([l_n]))_{n\in \omega}$ is bounded.
\end{enumerate}
\end{bigtheorem}

\noindent One consequence of Theorem \ref{lengthfunctiontheorem} is that $(*)$ implies that the class of strongly bounded groups (see \cite{dC}) is closed under countable products.

\end{section}

\begin{section}{Theorem \ref{thebigone} and some consequences}\label{TheoremA}
We give a brief review of some topological concepts (see \cite{Sr}).  If $X$ is a topological space we say $Y\subseteq X$ is \emph{nowhere dense} in $X$ provided the closure of $Y$ has empty interior.  A set $Y\subseteq X$ is \emph{meager} in $X$ if $Y$ is a union of countably many nowhere dense subsets of $X$, and $Y \subseteq X$ is \emph{comeager} in $X$ provided $X\setminus Y$ is meager in $X$.  Furthermore, $Y \subseteq X$ has the \emph{Baire property} provided there exists an open set $U\subseteq X$ such that the symmetric difference $Y \Delta U = (Y \cup U)\setminus (Y \cap U)$ is meager in $X$.  A space in which a countable intersection of dense open sets is dense is a \emph{Baire space}.  A nonempty open set in a Baire space cannot be meager.

A space is \emph{Polish} if it is separable and completely metrizable.  Polish spaces are second countable and are Baire spaces by Theorem \ref{Baire} below.  The Cantor space $2^{\omega} = \{0, 1\}^{\omega}$ consisting of all sequences of $0$s and $1$s, topologized by the product topology on the discrete space $\{0, 1\}$, is an example of a Polish space.  The metric $d(\sigma_0, \sigma_1) = (\min\{n\in \omega\mid \sigma_0(n) \neq \sigma_1(n)\})^{-1}$ is a complete metric which generates the topology on $2^{\omega}$ and the set of eventually $0$ sequences is a countable dense subset.  A basis for the topology is given by the collection of open sets of form $U(\langle s_0, \ldots, s_{m-1}\rangle) = \{\sigma\in 2^{\omega}\mid (\forall k<m)(\sigma(k) = s_k)\}$.  By interspersing coordinates it is easy to see that $2^{\omega}$ is homeomorphic to the space $2^{\omega} \times 2^{\omega}$.

We state some classical theorems that will be used in our proof of Theorem \ref{thebigone}.  As indicated, each of these results is a consequence of ZF + DC.  Theorem \ref{Baire} is standard; a proof can be found in \cite[Theorem 2.2.5]{Sr}.  Theorem \ref{KU} appeared in \cite{KU} and a proof can be found in \cite[Theorem 3.5.16]{Sr}.  Theorem \ref{Myc} essentially comes from \cite{M} and a short proof is given in \cite[Theorem 5.13.10]{Sr}.

\begin{theorem}\label{Baire} (ZF + DC)  Completely metrizable spaces are Baire spaces.
\end{theorem}

\begin{theorem}\label{KU} (ZF + DC)  Let $X, Y$ be second countable Baire spaces.  If $A \subseteq X\times Y$ has the Baire property then the following are equivalent:

\begin{enumerate}  \item $A$ is meager in $X\times Y$

\item $\{x\in X\mid \{y\in Y\mid (x, y)\in A\}\text{ is meager in } Y\}$ is comeager in $X$

\end{enumerate}  
\end{theorem}

\begin{theorem}\label{Myc} (ZF + DC)  If $X$ is a Polish space and $E \subseteq X\times X$ is a meager equivalence relation then $2^{\omega}$ injects into the set of equivalence classes of $E$.
\end{theorem}

\begin{definition}  We say an equivalence relation $E \subseteq 2^{\omega}\times 2^{\omega}$ satisfies $(\dagger)$ if $|\{n\in \omega: \sigma_0(n) \neq \sigma_1(n)\}| = 1$ implies $\neg (\sigma_0E\sigma_1)$ (see \cite{Sh2}.)  If $E$ is an equivalence relation on $2^{\omega}$ and $\sigma \in 2^{\omega}$ we let $E\sigma$ denote the equivalence class of $\sigma$.
\end{definition}

\begin{lemma}\label{equivclass} $(*)$ If $E\subseteq 2^{\omega}\times 2^{\omega}$ satisfies $(\dagger)$ then $|2^{\omega}|\leq |\{E_{\sigma}\}_{\sigma\in 2^{\omega}}|$.
\end{lemma}

\begin{proof}  Assume $(*)$ and that $E\subseteq 2^{\omega}\times 2^{\omega}$ is an equivalence relation which satisfies $(\dagger)$.  Suppose $E_{\sigma}$ is not meager for some $\sigma\in 2^{\omega}$.  By $(*)$, $E_{\sigma}$ has the Baire property and we obtain a basic open set $U = U(\langle s_0, \ldots, s_{n-1}\rangle) \subseteq 2^{\omega}$ for which $U\setminus E_{\sigma}$ is meager in $2^{\omega}$.  The involution homeomorphism $f:2^{\omega} \rightarrow 2^{\omega}$ obtained by switching the $n$-th coordinate has $f(U) =U$ and preserves meagerness.  Since $U \setminus E_{\sigma}$ is meager in $2^{\omega}$, so is $f(U \setminus E_{\sigma})=U\setminus f(E_{\sigma})$.  Then $(U\setminus E_{\sigma})\cup (U\setminus f(E_{\sigma})) = U \setminus (E_{\sigma}\cap f(E_{\sigma}))$ is also meager in $2^{\omega}$.  However, $(\dagger)$ implies $E_{\sigma}\cap f(E_{\sigma}) =\emptyset$, so that $U$ is meager in $2^{\omega}$.  This contradicts the Baire category theorem for complete metric spaces.

Thus each equivalence class $E_{\sigma}$ is meager, and since $E$ has the Baire property in $2^{\omega}\times 2^{\omega}$ we see that $E$ is meager by Theorem \ref{KU}.  By Theorem \ref{Myc}, we get $|2^{\omega}|\leq |\{E_{\sigma}\}_{\sigma\in 2^{\omega}}|$.
\end{proof}

\begin{definition} \label{Hausdorffarrow} Given a sequence of group elements $(g_n)_{n\in \omega}$ in a Hausdorff topological group $G$ we write $g_n \searrow 1_G$ if for each $\sigma \in 2^{\omega}$ the sequence $(g_0^{\sigma(0)}\cdots g_{n-1}^{\sigma(n-1)})_{n\in \omega}$ converges.  By the Hausdorff condition, the limit of such a sequence is unique.  When $g_n\searrow 1_G$ and $\sigma \in 2^{\omega}$ we write $g^{\sigma} = \lim_{n\rightarrow \infty} g_0^{\sigma(0)}\cdots g_{n-1}^{\sigma(n-1)}$.
\end{definition}

\begin{definition} \label{Looparrow} Given a topological space $X$ and a sequence of loops $(l_n)_{n\in \omega}$ based at $x\in X$ we write $l_n \searrow x$ if for every neighborhood $U$ of $x$ the $l_n$ eventually have images inside of $U$.  In other words, $(\forall U)(\exists N\in \omega)(\forall n\geq N)(l_n([0, 1])\subseteq U)$ where $U$ is taken to range over open neighborhoods of $x$.  Given a loop $l$ at $x$ we define for $i\in \{0, 1\}$ the loop $l^i:[0, 1]\rightarrow X$ to be the constant loop at $x$ if $i = 0$ and to be $l$ otherwise.  Given a sequence $(l_n)_{n\in \omega}$ such that $l_n \searrow x$ and $\sigma\in 2^{\omega}$ we obtain a loop $l^\sigma$ by letting 

\begin{center}
$l^{\sigma}(t) = \begin{cases} l_n^{\sigma(n)}(2^{n+1}(t-1)+2)$ if $t\in [1-\frac{1}{2^n}, 1-\frac{1}{2^{n+1}}]\\x$ if $t = 1 \end{cases}$
\end{center}

\noindent The loop $l^{\sigma}$ may be regarded as an infinite concatenation of the sequence of loops $(l_n^{\sigma(n)})_{n\in \omega}$ since it passes through the loop $l_0^{\sigma(0)}$ over the interval $[0, \frac{1}{2}]$, through the loop $l_1^{\sigma(1)}$ over the interval $[\frac{1}{2}, \frac{3}{4}]$, etc.  If $l_n\searrow x$ and $K \leq \pi_1(X, x)$ we say $(l_n)_{n\in \omega}$ is \textit{deeply concatenable} for $K$ provided $[l^{\sigma}]\in K$ for all $\sigma\in 2^{\omega}$.
\end{definition}

\begin{proof}(of Theorem \ref{thebigone})  We prove (2) and the claim in (1) follows along almost precisely the same argument.  Assume the hypotheses.  Any subsequence of $(l_n)_{n\in \omega}$ must also be deeply concatenable for $K\leq \pi_1(X, x)$, and so without loss of generality we may assume that $[l_n]\notin \ker(\phi)$ for all $n\in \omega$.  Define $E \subseteq 2^{\omega} \times 2^{\omega}$ by letting $\sigma_0E\sigma_1$ if and only if $\phi([l^{\sigma_0}]) = \phi([l^{\sigma_1}])$.  We show $E$ satisfies $(\dagger)$ and conclude by applying Lemma \ref{equivclass}.

Supposing $\sigma_0$ and $\sigma_1$ differ at precisely the $n$-th coordinate and letting $\tau\in \{0, 1\}^{\omega}$ be given by $\tau(m) = \sigma(n+m+1)$ we notice that 

\begin{center}  $\phi([l^{\sigma_0}]) = \phi([l_0^{\sigma_0(0)}*\cdots *l_{n-1}^{\sigma_0(n-1)}])\phi([l_n^{\sigma_0(n)}])\phi([l^{\tau}])$
\end{center}

\noindent and

\begin{center}
$\phi([l^{\sigma_1}]) =\phi([l_0^{\sigma_1(0)}*\cdots* l_{n-1}^{\sigma_1(n-1)}])\phi([l_n^{\sigma_1(n)}])\phi([l^{\tau}])$

 $= \phi([l_0^{\sigma_0(0)}*\cdots *l_{n-1}^{\sigma_0(n-1)}])\phi([l_n^{\sigma_1(n)}])\phi([l^{\tau}])$
\end{center}

\noindent and if $\sigma_0E\sigma_1$ we cancel on the left and right by $\phi([l_0^{\sigma_0(0)}*\cdots* l_{n-1}^{\sigma_0(n-1)}])$ and $\phi([l^{\tau}])$ respectively and see that $\phi([l_n^{\sigma_0(n)}]) = \phi([l_n^{\sigma_1(n)}])$.  Thus $[l_n]\in \ker(\phi)$, contrary to assumption.  Then $E$ satisfies $(\dagger)$.
\end{proof}

\begin{theorem}\label{cmgroupsopenkernel} $(*)$  If $G$ is a completely metrizable group and $\phi: G \rightarrow H$ does not have open kernel then $|2^{\omega}| \leq |H|$.
\end{theorem}

\begin{proof}  Let $\phi:G \rightarrow H$ be a homomorphism such that $\ker(\phi)$ is not open.  Let $d$ be a complete metric compatible with the topoogy of $G$.  Let $g_0 \in G\setminus \ker(\phi)$.  Select a neighborhood $U_1$ of $1_G$ such that $g\in U_1$ implies both

\begin{center}  $d(g_0g, g_0) \leq \frac{1}{2}$
\end{center}

\noindent and

\begin{center}  $d(g, 1_G) \leq \frac{1}{2}$
\end{center}

Select $g_1\in U_1\setminus \ker(\phi)$.  Select open neighborhood $U_2$ of $1_G$ such that $g\in U_2$ implies all of the following:

\begin{center}  $d(g_0g_1 g, g_0g_1)\leq\frac{1}{4}$

$d(g_0g, g_0)\leq \frac{1}{4}$

$d(g_1g, g_1)\leq \frac{1}{4}$

$d(g, 1_G) \leq \frac{1}{4}$
\end{center}

\noindent and select $g_2 \in U_2\setminus \ker(\phi)$.  In general, assuming we have selected $g_0, \ldots, g_{n-1}$ and $U_1, \cdots, U_{n-1}$ in this way, we select open neighborhood $U_n$ of $1_G$ such that $g\in U_n$ implies that for all $\langle s_0, s_1, \ldots, s_{n-1}\rangle$ with $s_i \in \{0, 1\}$ we get

\begin{center}
$d(g_0^{s_0}g_1^{s_1}\cdots g_{n-1}^{s_{n-1}}g, g_0^{s_0}g_1^{s_1}\cdots g_{n-1}^{s_{n-1}})\leq 2^{-n}$
\end{center}

\noindent and select $g_n\in U_n \setminus \ker(\phi)$.  For every $\sigma\in 2^{\omega}$ the sequence $(g_0^{\sigma(0)}\cdots g_n^{\sigma(n)})_{n\in \omega}$ is Cauchy, and therefore convergent.  Then $g_n\searrow 1_G$ and $g_n \notin \ker(\phi)$.  We finish by applying Theorem \ref{thebigone} part (1).
\end{proof}

A further consequence of Theorem \ref{thebigone} is the following:
 
\begin{corollary}\label{fundamentalgroup} $(*)$ Suppose $X$ is a space which is first countable at $x\in X$.  Suppose also that for the homomorphism $\phi: \pi_1(X, x) \rightarrow H$ there is no neighborhood $U$ of $x$ for which each loop $l$ at $x$ in $U$ satisfies $[l]\in \ker(\phi)$.  Then $|2^{\omega}| \leq |H|$.
\end{corollary}

\begin{proof}  Assume the hypotheses.  Let $(U_n)_{n\in \omega}$ be a basis of open neighborhoods at $x$.  For each $n\in\omega$ select $l_n$ with image in $U_n$ and $[l_n]\notin\ker(\phi)$.  We have $l_n \searrow x$ and certainly $(l_n)_{n\in \omega}$ is deeply concatenable for $\pi_1(X, x)$.  Apply Theorem \ref{thebigone} part (2).
\end{proof}

James Cannon and Greg Conner have asked the following:
\vspace{.15in}

\emph{If $X$ is a separable, locally path connected metric space and there exists $x\in X$ at which there is an essential loop which may be homotoped fixing the endpoints at $x$ to have arbitrarily small diameter, is it the case that $\pi_1(X, x)$ is uncountable?}
\vspace{.15in}

\noindent Though this problem is open as of this writing even in the case $X$ is Borel, we get the following strengthening of the affirmative answer under $(*)$:

\begin{corollary}\label{CannonConner}$(*)$  If $X$ is first countable at $x$ and each neighborhood of $x$ contains an essential loop at $x$, then $|2^{\omega}| \leq |\pi_1(X, x)|$.
\end{corollary}

\begin{proof}  Apply Corollary \ref{fundamentalgroup} to the identity map $\pi_1(X, x) \rightarrow \pi_1(X, x)$.
\end{proof}

\end{section}

\begin{section}{Unfreeness and slenderness}\label{FromA}

We give further consequences of Theorem \ref{thebigone} which are related to algebraic freeness and to slenderness.  Under the algebraic freeness theme, a subgroup of $\mathbb{Z}^{\omega}$ as well as two subgroups of the Hawaiian earring group are shown not to be free if one assumes $(*)$.  Similarly, numerous subspaces of $\mathbb{Q}^{\omega}$ are shown not to have a Hamel basis.  Assumption $(*)$ also implies that (noncommutativeley) slender groups must be uncountable.  We begin with the straightforward:

\begin{proposition}\label{nosmallretract}  $(*)$ The following hold:
\begin{enumerate}

\item  If $G$ is a Hausdorff topological group, $g_n\searrow 1_G$ with $g_n\neq 1_G$ for all $n\in \omega$, and $H\leq G$ is countable such that $\{g_n\}_{n\in \omega} \subseteq H$, then $H$ is not a retract of $G$.

\item  If $(l_n)_{n\in \omega}$ is deeply concatenable for $K\leq \pi_1(X, x)$, the $[l_n]$ are each nontrivial, and $H \leq K$ is countable such that $\{[l_n]\}_{n\in \omega} \subseteq H$, then $H$ is not a retract of $K$.

\end{enumerate}
\end{proposition}

\begin{proof}  In either case the retraction would imply a contradiction to Theorem \ref{thebigone}.
\end{proof}

We recall the definition of a graph product of groups (see \cite{G}).  Let $\Gamma = (V, E)$ be a graph (we allow the sets of vertices and edges to be of arbitrary cardinality but do not allow an edge to connect a vertex to itself) and to each vertex $v\in V$ we associate a group $G_v$.  The graph product $\Gamma(\{G_v\}_{v\in V})$ is defined by taking the quotient of the free product $\ast_{v\in V} G_v$ by the normal closure of the set $\{[g_{v_0}, g_{v_1}]\}_{g_{v_0} \in G_{v_0}, g_{v_1} \in G_{v_1}, \{v_0, v_1\}\in E}$.  Thus free products of groups and direct sums of groups are examples of graph products of groups, with the graphs having either no edges or being complete in the respective cases.  A graph product is a \emph{right angled Artin group} if $G_v\simeq \mathbb{Z}$ for all $v\in V$.

\begin{corollary}  \label{nographproduct} $(*)$ The following hold:

\begin{enumerate}
\item  If $G$ is a Hausdorff topological group, $g_n\searrow 1_G$ with $g_n\neq 1_G$ for all $n\in \omega$, then $G$ is not a graph product of countable groups.

\item  If $(l_n)_{n\in \omega}$ is deeply concatenable for $K\leq \pi_1(X, x)$ and the $[l_n]$ are each nontrivial, then $K$ is not a graph product of countable groups.

\item   Suppose $\{H_m\}_{m\in \omega}$ is a collection of groups and $(h_m)_m \in \prod_{m\in \omega}H_m$ is such that $h_m \neq 1_{H_m}$.  There does not exist a right angled Artin group $K \leq \prod_{m\in \omega}H_m$ including the set of sequences $\{(h_m^{\sigma(m)})_{m\in \omega}\}_{\sigma\in 2^{\omega}}$.
\end{enumerate}
\end{corollary}

\begin{proof}  We prove the claim in (2) and the proof for (1) is completely analogous.  Assume the hypotheses for (2) and imagine for contradiction that $K$ is a graph product of countable groups, say $K \simeq  \Gamma(\{G_v\}_{v\in V})$ with each $G_v$ countable.  For each $n\in \omega$ we have a finite subset $V_n\subseteq V$ for which $[l_n]  \in \Gamma(\{G_v\}_{v\in V_n})$.  The subgroup $\Gamma(\{G_v\}_{v\in \bigcup_{n\in \omega}V_n})$ is a countable retract of $K$, contradicting Proposition \ref{nosmallretract}.

For (3) we endow $\prod_{m\in \omega}H_m$ with the topology of coordinate wise convergence (a sequence of elements in $\prod_{m\in \omega}H_m$ converges if and only if each coordinate eventually stabilizes).  Define $g_n \in \prod_{m\in \omega}H_m$ by
\begin{center}
$g_n(m) = \begin{cases} h_n$ if $m = n\\1_{H_{m}}$ otherwise $ \end{cases}$
\end{center}
\noindent  Supposing $K$ to be a right-angled Artin group including $\{(h_m^{\sigma(m)})_{m\in \omega}\}_{\sigma\in 2^{\omega}}$, we endow $K$ with the subspace group topology.  It is clear that $g_n \searrow 1_K$ in $K$ and $g_n \neq 1_K$ and we obtain a contradiction to (1).
\end{proof}

Next we consider the subgroup $\mathcal{B}\subseteq \mathbb{Z}^{\omega}$ whose elements are the bounded sequences.  That this group is free abelian follows from the axiom of choice (see \cite{No}).  Specker had an earlier proof that this group is free abelian, but assumed the continuum hypothesis \cite{Sp}. The group $\mathbb{Z}^{\omega}$ is itself not free abelian, and this can be proved from ZF + DC (see \cite{Sp} or \cite{F}).  The next result was first given by Blass \cite{Bl}:

\begin{theorem}\label{Blass}  $(*)$ The group $\mathcal{B}$ is not free abelian.
\end{theorem}

\begin{proof}  Consider the element $(1)_{m\in \omega} \in \mathbb{Z}^{\omega}$ (here we are using additive notation) and apply Corollary \ref{nographproduct} (3).
\end{proof}

In this same spirit, we give a result expressing the inability to select a Hamel basis for certain vector spaces over $\mathbb{Q}$.

\begin{theorem}\label{noHamel}  $(*)$  Let $f \in \mathbb{Q}^{\omega}$ be a positive function.  The following groups cannot be expressed as a direct sum $\bigoplus_{i\in I}\mathbb{Q}$:

\begin{enumerate}

\item $\mathbb{Q}^{\omega}$

\item The subgroup in $\mathbb{Q}^{\omega}$ of all sequences $(h_m)_{m\in \omega}$ for which $|h_m|$ is $O(f)$

\item The subgroup in $\mathbb{Q}^{\omega}$ of all sequences $(h_m)_{m\in \omega}$ for which $|h_m|$ is $o(f)$

\end{enumerate}
\end{theorem}

\begin{proof}  We prove (3) and the same proof applies to  (1) and (2).  Let the group in question be denoted $G$.  For each $n\in \omega$ let $g_n = \begin{cases} f(m)/m$ if $m = n\\ 0$ otherwise $\end{cases}$.  Endowing $G \leq \mathbb{Q}^{\omega}$ with the topology of coordinate wise convergence with each coordinate having the discrete topology, we see that $g_n\searrow 0_G$ and $g_n \neq 0_G$ for all $n\in \omega$.  Apply Corollary \ref{nographproduct} (1).
\end{proof}

We give some background on the Hawaiian earring group so as to state and prove a non-abelian analogue of Theorem \ref{Blass}. The interested reader may consult \cite{CC} or \cite{E} for a more in depth exposition of the Hawaiian earring group.  Let $\{a_n^{\pm 1}\}_{n\in \omega}$ be a countably infinite set with formal inverses.  A function $W: \overline{W}\rightarrow \{a_n^{\pm 1}\}_{n\in \omega}$ from a countable totally ordered set $\overline{W}$ is a \emph{word} if for each $n\in \omega$ the set $W^{-1}(\{a_n^{\pm 1}\})$ is finite.  For two words $U$ and $V$ we write $U \equiv V$ if there exists an order isomorphism of the domains of each word $f: \overline{U} \rightarrow \overline{V}$ such that $U(i) = V(f(i))$.  We use $\W$ to denote a selection from each $\equiv$ equivalence class.  For each $N\in \omega$ there is a projection map $p_N$ to the set of finite words given by letting $p_N(W) = W\upharpoonright \{i\in \overline{W}: W(i) \in \{a_n^{\pm 1}\}_{n=0}^{N}\}$.  For words $U, V\in \W$ we let $U \sim V$ if for each $N\in \omega$ we have $p_N(U) = p_N(V)$ in the free group $F(\{a_0, \ldots, a_N\})$.  This is an equivalence relation.  For each word $U$ there is an inverse word $U^{-1}$ whose domain is the totally ordered set $\overline{U}$ under the reverse order and $U^{-1}(i) = (U(i))^{-1}$.  Given two words $U, V\in \W$ there is a natural way to form the concatenation $UV$.  One takes the domain of $UV$ to be the disjoint union of $\overline{U}$ with $\overline{V}$, with order extending that of $\overline{U}$ and $\overline{V}$ and placing all elements of $\overline{U}$ before those of $\overline{V}$, and $UV(i) = \begin{cases}U(i)$ if $i\in \overline{U}\\V(i)$ if $i\in \overline{V}  \end{cases}$.  The set $\W/\sim$ now has a group structure with binary operation $[U][V] = [UV]$, inverses defined by $[U]^{-1} = [U^{-1}]$ and the trivial element given by the equivalence class of the empty word $E$.

Let $\HEG$ denote the group $\W/\sim$.  The free group $F(\{a_0, \ldots, a_N\})$, which we shall denote $\HEG_N$, may be though of as a subgroup in $\HEG$ in the obvious way.  Moreover, the word map $p_N$ defines a group retraction $\HEG \rightarrow \HEG_N$ which we denote $p_N$ by abuse of notation.  A word $W\in \W$ is \emph{reduced} if $W\equiv UVX$ with $V \sim E$ implies $V\equiv E$.  Each class $[U]\in\HEG$ has a unique reduced representative $W$ \cite[Theorem 1.4]{E}.

We highlight two subgroups of $\HEG$.  The first subgroup, denoted $\Bd$, consists of those elements $[W] \in \HEG$ for whose reduced word representative $W$ there exists a number $N\in \omega$ such that for every $n\in \omega$ we have $|\{i\in \overline{W}\mid W(i) \in \{a_n^{\pm 1}\}|<N$.  This group was introduced by Zastrow in \cite{Z} as an analogue to the abelian group $\mathcal{B}$ and shown to be a free group using the axiom of choice.  The second subgroup of $\HEG$, which we denote $\Scat$, consists of those elements of $\HEG$ whose reduced word representative has domain of scattered order type (i.e. contains no densely ordered subset with more than one element).  This subgroup was studied in \cite{CC} and shown therein to be a free group under ZFC.

\begin{theorem}$(*)$  Neither $\Bd$ nor $\Scat$ is free.
\end{theorem}

\begin{proof}  Let $\Red(W)$ denote the finite reduced word representative of an element $[W]\in \HEG_N$.  By how the relation $\sim$ was defined, each element $[W] \in \HEG$ is uniquely determined by the sequence $(\Red(p_N(W)))_{N\in \omega}$.  Letting $\overline{F}$ be the inverse limit $\varprojlim \HEG_N$ given by homomorphisms $p_{N+1}: \HEG_{N+1} \rightarrow \HEG_N$, we immediately see that $\HEG$ is a subgroup of $\overline{F}$ via the monomorphism $[W] \mapsto (\Red(p_N(W)))_{N\in \omega}$.  Topologize $\overline{F}$ by coordinate wise convergence.

For each $n\in\omega$ let $g_n \in \HEG$ be the word $a_n$, so $g_n$ is represented in $\overline{F}$ by the sequence $(U_{N, n})_{N\in \omega}$ where $U_{N, n} \equiv \begin{cases}a_n$ if $n\geq N\\ E$ otherwise $\end{cases}$.   It is easy to see that $g_n \neq 1_{\overline{F}}$ for all $n$ and that $g_n \searrow 1_{\overline{F}}$; for each $\sigma\in 2^{\omega}$ the element $g^{\sigma}$ is represented by the word of either $\omega$ or finite order type given by $a_0^{\sigma(0)}a_1^{\sigma(1)}\cdots$.  In particular each such word $g^{\sigma}$ is an element of $\Scat$ and of $\Bd$.  We finish by applying Corollary \ref{nographproduct} (1).
\end{proof}

Recall that an abelian group $A$ is \emph{slender} if for every homomorphism $\phi: \mathbb{Z}^{\omega} \rightarrow A$ there exists $N\in \omega$ such that $\phi = \phi \circ \proj_N$ where $\proj_N: \mathbb{Z}^{\omega} \rightarrow \prod_{n=0}^{N-1}\mathbb{Z} \times (0)_{n=N}^{\infty}$ is the retraction which projects to the first $N$ coordinates.  It is a classically known consequence of ZFC that a group is slender if and only if it is torsion-free and does not contain an isomorph of $\mathbb{Q}$, $\mathbb{Z}^{\omega}$ or the $p$-adic integers $J_p$ for some prime $p$ (see \cite{Nu}).

More generally a group $H$ is \emph{noncommutatively slender} (or \emph{n-slender}) provided for every homomorphism $\phi: \HEG \rightarrow H$ there exists $N\in \omega$ such that $\phi = \phi\circ p_N$, where $p_N: \HEG \rightarrow \HEG_N$ is the retraction described earlier.  Free groups \cite{H}, Baumslag-Solitar groups \cite{Na} and torsion-free word hyperbolic groups \cite{Co} are are examples of n-slender groups.  Assuming the axiom of choice, n-slender groups can not include torsion or a subgroup isomorphic to $\mathbb{Q}$ \cite{E}.

\begin{theorem}\label{slender} $(*)$  If $A$ fails to be slender then $|2^{\omega}| \leq |A|$.  Also, if $H$ fails to be n-slender then $|2^{\omega}| \leq |H|$.

\end{theorem}

\begin{proof}  For the first claim, we notice that the group $\mathbb{Z}^{\omega}$ is a Polish group under coordinate wise convergence.  Supposing $A$ is not slender we get a homomorphism $\phi:\mathbb{Z}^{\omega} \rightarrow A$ which does not have open kernel.  By Theorem \ref{cmgroupsopenkernel} we see $|2^{\omega}| \leq |A|$.

For the second claim, suppose $H$ is not n-slender and let $\phi: \HEG \rightarrow H$ witness.  Topologize $\HEG$ under coordinate wise convergence.  For each $n\in \omega$ we get an element $(W_{N, n})_{N\in \omega}\in \HEG\leq \overline{F}$ such that $W_{N, n}= E$ for $N<n$ and $(W_{N, n})_{N\in \omega}\notin \ker(\phi)$.  Now $(W_{N, n})_{N\in \omega}\searrow E$ and by Theorem \ref{thebigone} (1) we have $|2^{\omega}| \leq |H|$.
\end{proof}

One can similarly say a group $H$ is \emph{completely metrizable slender} (or \emph{cm-slender}) if any homomorphism from a completely metrizable group to $H$ has open kernel \cite{CoCo}.  Theorem \ref{cmgroupsopenkernel} states precisely that if $H$ is not cm-slender then $|2^{\omega}| \leq |H|$.
\end{section}

\begin{section}{Deep boundedness for length functions}\label{Length}

We restate and prove Theorem \ref{lengthfunctiontheorem}, first recalling the following definition:

\begin{definition} \label{length}  A \emph{length function} on a group $G$ is a function $L: G\rightarrow [0, \infty)$ such that

\begin{enumerate}\item  $L(1_G) = 0$

\item  $L(g) = L(g^{-1})$

\item  $L(gh) \leq L(g) + L(h)$
\end{enumerate}
\end{definition}

\begin{B} $(*)$ The following hold:
\begin{enumerate}  
\item  Suppose $G$ is a Hausdorff topological group, $g_n\searrow 1_G$ and $L$ is a length function on $G$.  The sequence $(L(g_n))_{n\in \omega}$ is bounded.

 \item  Suppose $(l_n)_{n\in \omega}$ is deeply concatenable for $K\leq \pi_1(X, x)$ and $L$ is a length function on $K$.  The sequence $(L([l_n]))_{n\in \omega}$ is bounded.
\end{enumerate}
\end{B}

\begin{proof}  We prove part (2) and the proof of part (1) is entirely analogous.  Suppose the conclusion fails.  By selecting a subsequence if necessary, we may assume that $L([l_0]) > 1$ and that $L([l_n]) > n + \sum_{m=0}^{n-1}L([l_m])$.  Letting $J_n = \{\sigma\in 2^{\omega}\mid L([l^{\sigma}]) \leq n\}$, we have $\bigcup_{n\in \omega}J_n = 2^{\omega}$.  As the $J_n$ all have the Baire property and $\bigcup_{n\in \omega}J_n = 2^{\omega}$, there exists some $J_n$ which is not meager.  Select a basic open set $U = U(s_0, \ldots, s_{k-1})$ such that $U\setminus E_n$ is meager in $2^{\omega}$.  Let $\mathcal{F}$ denote the set of finite subsets of $[k, \infty) \cap \omega$.  For each $F \in \mathcal{F}$ we let $f_F: 2^{\omega} \rightarrow 2^{\omega}$ be the homeomorphism that switches precisely the coordinates in $F$.  Certainly $f_F(U) = U$ for all $F \in \mathcal{F}$.  We have $U \setminus f_F(E_n)$ meager for each $F \in \mathcal{F}$.  Then $U \setminus (\bigcap_{F \in \mathcal{F}}f_F(E_n))$ is meager in $2^{\omega}$ as a countable union of meager sets.  As $U$ is not meager, we see that $U \cap\bigcap_{F\in \mathcal{F}}f_F(E_n)$ is nonempty and let $\sigma$ be an element thereof.

Select $N>2n  +  \sum_{m=0}^{k-1}L([l_m])$.  As $f_{\{N\}}(\sigma) \in J_n$ we may assume without loss of generality that $\sigma(N) = 1$.  Let $\tau\in 2^{\omega}$ be the sequence $\tau(m) =\begin{cases} 0$ if $m<N+1\\ \sigma(m)$ if $N+1 \leq m \end{cases}$.  Notice that $[l^{\sigma}] = [l_0^{\sigma(0)}*\cdots* l_{N-1}^{\sigma(N-1)}][l_N][l^{\tau}]$.  Let $F = \{k\leq q<N\mid\sigma(q) = 1\}$.  We also have

\begin{center}
$[l^{\tau}] = [l_0^{\sigma(0)}*\cdots *l_{k-1}^{\sigma(k-1)}]^{-1}[l^{f_F(\sigma)}]$
\end{center}

\noindent and therefore 

\begin{center}
$L([l^{\tau}]) \leq L([l_0^{\sigma(0)}*\cdots *l_{k-1}^{\sigma(k-1)}]) + n$

$\leq \sum_{m=0}^{k-1}L([l_m]) +n$
\end{center}

\noindent Thus it follows that

\begin{center} $N +  \sum_{m=0}^{N-1}L([l_m])< L([l_N])$

$\leq  L([l_0^{\sigma(0)}*\cdots* l_{N-1}^{\sigma(N-1)}]) + L([l^{\sigma}])+ L([l^{\tau}])$

$\leq \sum_{m = 0}^{N-1}L([l_m]) + n +  \sum_{m=0}^{k-1}L([l_m]) +n$
\end{center}

\noindent whence $N < 2n +   \sum_{m=0}^{k-1}L([l_m])$, a contradiction.

\end{proof}

We point out that the sequence $(L(g_n))_{n\in\omega}$ in Theorem \ref{lengthfunctiontheorem} needn't go to zero, even assuming $(*)$.  To see this simply consider the length function on $\mathbb{R}$ which assigns length $1$ to every non-identity element.

\begin{theorem}\label{boundforcm}$(*)$  If $G$ is a completely metrizable topological group and $L: G \rightarrow [0, \infty)$ a length function, there exists an open neighborhood $U\subseteq G$ of $1_G$ and $N\in \omega$ such that $g\in U$ implies $L(g) \leq N$.
\end{theorem}

\begin{proof}  The proof will be similar to that of Theorem \ref{cmgroupsopenkernel}.  Let $d$ be a complete metric compatible with the topology on $G$.  Suppose the conclusion is false.  Select $g_0\in G$ such that $L(g_0)>1$.  Pick a neighborhood $U_1$ of $1_G$ for which $g\in U_1$ implies both

\begin{center}
$d(g_0g, g_0) \leq\frac{1}{2}$

$d(g, 1_G)\leq \frac{1}{2}$

\end{center}

\noindent  Pick $g_1 \in U_1$ such that $L(g_1) > 2$.  In general, assuming we have selected $g_0, \ldots, g_{n-1}$ and $U_1, \cdots, U_{n-1}$ in this way, we select open neighborhood $U_n$ of $1_G$ such that $g\in U_n$ implies that for all $\langle s_0, s_1, \ldots, s_{n-1}\rangle$ with $s_i \in \{0, 1\}$ we get

\begin{center}
$d(g_0^{s_0}g_1^{s_1}\cdots g_{n-1}^{s_{n-1}}g, g_0^{s_0}g_1^{s_1}\cdots g_{n-1}^{s_{n-1}})\leq 2^{-n}$
\end{center}

\noindent and select $g_n\in U_n$ such that $L(g_n) > n+1$.  For every $\sigma\in 2^{\omega}$ the sequence $(g_0^{\sigma(0)}\cdots g_n^{\sigma(n)})_{n\in \omega}$ is Cauchy, and therefore convergent.  Then $g_n\searrow 1_G$.  We finish by applying Theorem \ref{lengthfunctiontheorem} part (1) for a contradiction.
\end{proof}

Theorem \ref{boundforcm} allows us to show that no infinite product over $\mathbb{Q}$ has a Hamel basis when $(*)$ is assumed, improving Theorem \ref{noHamel} part (1):

\begin{corollary}$(*)$ If $X$ is infinite then $\mathbb{Q}^X$ is not isomorphic to a direct sum $\bigoplus_{i\in I}\mathbb{Q}$.
\end{corollary}

\begin{proof}  Suppose for contradiction that $\phi: \mathbb{Q}^X\rightarrow \bigoplus_{i\in I}\mathbb{Q}$ is an isomorphism.  By DC we let $f:\omega \rightarrow X$ be an injection, which in turn gives a monomorphism $\phi_f: \mathbb{Q}^{\omega} \rightarrow \mathbb{Q}^X$ by letting a sequence $\sigma: \omega \rightarrow \mathbb{Q}$ be taken $\phi_f(\sigma): X \rightarrow \mathbb{Q}$ defined by $\phi_f(\sigma)(x) = \begin{cases}\sigma(n)$ if $x = f(n)\\ 0$ otherwise $\end{cases}$.  Defining $L: \mathbb{Q}^{\omega} \rightarrow [0, \infty)$ to be the sum of the absolute values of the coordinates of $\phi\circ\phi_{f}$ we get that $L$ is a length function, and in fact a vector space norm.  Topologizing each coordinate discretely we topologize $\mathbb{Q}^{\omega}$ by coordinatewise convergence we have that this group topology is completely metrizable.  By Theorem \ref{boundforcm} we have a neighborhood of identity whose elements have a uniformly bounded length.  Then for some $N, M\in \omega$ every element of $\mathbb{Q}^{\omega}$ whose first $N$ coordinates are $0$ has length at most $M$.  But this requires that the sequence $\sigma \in \mathbb{Q}^{\omega}$ whose first $N$ coordinates are 0 and the rest of whose coordinates are $1$ must have length $0$, for otherwise we could select a $P \in \omega$ large enough that $P \sigma$ has length greater than $M$.  This is a contradiction to the fact that $L$ is a vector space norm.
\end{proof}

\begin{corollary}\label{cofinality}$(*)$  If $G$ is a completely metrizable topological group and $(G_n)_{n\in\omega}$ is a sequence of subgroups such that $G_n \leq G_{n+1}$ and $\bigcup_{n\in \omega}G_n = G$, then the $G_n$ are eventually clopen.
\end{corollary}

\begin{proof}  Assume the hypotheses.  Let $L: G \rightarrow [0, \infty)$ be the length function given by $L(g) = \min\{n\in \omega\mid g\in G_n\}$.  By Theorem \ref{boundforcm} there exists a neighborhood $U$ of $1_G$ and $N\in \omega$ for which $g\in U$ implies $L(g) \leq N$.  Then $U \subseteq G_N$.  Since $G_N = \bigcup_{g\in G_N}gU$, we have $G_N$ open, and considering the left cosets of $G_N$ we see that $G\setminus G_N$ is also open.  Similarly $n\geq N$ implies $G_n$ is clopen.
\end{proof}

From Corollary \ref{cofinality} we see that under $(*)$ a countable product of finitely generated groups does not have $\omega$-cofinality.  In particular $\mathbb{Z}^{\omega}$ would not have $\omega$-cofinality. In contrast, the axiom of choice produces an epimorphism from $\mathbb{Z}^{\omega}$ to $\mathbb{Q}$.

We have the following direct corollary to Theorem \ref{lengthfunctiontheorem}:

\begin{corollary}\label{boundforpione} $(*)$  Let $X$ be first countable at $x\in X$ and $L: \pi_1(X, x) \rightarrow [0, \infty)$ be a length function.  There exists a neighborhood $U$ of $x$ and $N \in \omega$ such that for any loop $l$ at $x$ with image in $U$ we have $L([l]) \leq N$.
\end{corollary}

\begin{proof}  Let $(U_n)_{n\in \omega}$ be a nesting basis of open neighborhoods of $x$.  Supposing the claim is false, we obtain a sequence of  loops $(l_n)_{n\in \omega}$ at $x$ with the image of $l_n$ included in $U_n$ and $L([l_n]) \geq n+1$.  Certainly $(l_n)_{n\in\omega}$ is deeply concatenable for $\pi_1(X, x)$ and we get a contradiction to Theorem \ref{lengthfunctiontheorem} part (2).
\end{proof}

Corollary \ref{boundforpione} implies the following fact:

\begin{corollary}\label{niceneighborhood}$(*)$  Suppose $X$ is first countable at $x$ and $L: \pi_1(X, x) \rightarrow [0, \infty)$ is a lengh function such that $[l] \neq 1$ implies that for every $n\in \omega$ there exists $k_n\in \omega$ such that $L([l]^{k_n}) \geq n$.  Then there exists a neighborhood $U$ of $x$ such that any loop at $x$ with image in $U$ has length $0$.
\end{corollary}

\begin{proof}  Assume the hypotheses and let $U$ and $N\in \omega$ be as in the conclusion of Theorem \ref{boundforpione}.  If $l$ is a loop at $x$ with image in $U$ we must have $L([l]) = 0$, else we select $k_{N+1}$  for $[l]$ as in the hypotheses.  Concatenating $l$ with itself $k_{N+1}$ times gives a loop $l'$ in $U$ at $x$ such that $L([l']) = L([l]^{k_{N+1}}) \geq N+1 > N$, a contradiction.
\end{proof}

\begin{example}  Assuming $(*)$, if $X$ is first countable at $x$ and $\phi: \pi_1(X, x) \rightarrow \mathbb{R}$ is a homomorphism there exists a neighborhood $U$ at $x$ such that for any loop $l$ at $x$ with image in $U$ we have $\phi([l]) = 0$.  This is immediate from Corollary \ref{niceneighborhood} by the length function $L([l]) = \|\phi([l])\|$.

\end{example}

Recall that a group $G$ is \emph{strongly bounded} if every action of $G$ by isometries on a metric space has bounded orbits (see \cite{Be} or \cite{dC}).

\begin{theorem}\label{boundedproduct} $(*)$ If $(G_n)_{n\in \omega}$ is a sequence of strongly bounded groups then $\prod_{n\in \omega}G_n$ is also strongly bounded.
\end{theorem}

\begin{proof}  We give the group $G = \prod_{n\in \omega}G_n$ the topology of coordinate wise convergence with each $G_n$ discrete.  This topological group has a compatible complete metric given by $d(h_0, h_1) = (\min\{n\in \omega\mid h_0(n) \neq h_1(n)\})^{-1}$ for sequences $h_0, h_1 \in \prod_{n\in \omega}G_n$.  Supposing $G$ acts on metric space $(X, D)$, we select $x\in X$ and let $L(g) = D(x, gx)$.  By Theorem \ref{boundforcm} we have an open neighborhood $U$ of $1_G$ and $N\in \omega$ for which $h\in U$ implies $L(h) \leq N$.  Select $K \in \omega$ large enough that $(1_{G_n})_{n=0}^{K-1}\times \prod_{n = K}^{\infty}G_n \subseteq U$.  Since $(1_{G_n})_{n=0}^{K-2} \times G_{K-1} \times (1_{G_n})_{n = K}^{\infty}$ is strongly bounded there exists $N' \in \omega$ such that $L(h) \leq N'$ whenever $h\in (1_{G_n})_{n=0}^{K-2} \times G_{K-1} \times (1_{G_n})_{n = K}^{\infty}$.

Given $h\in (1_{G_n})_{n=0}^{K-2}\times \prod_{n = K}^{\infty}G_n$ we can write $h$ uniquely as $h = h_0h_1 = h_1h_0$ where $h_0\in (1_{G_n})_{n=0}^{K-2} \times G_{K-1} \times (1_{G_n})_{n = K}^{\infty}$ and $h_1 \in (1_{G_n})_{n=0}^{K-1}\times \prod_{n = K}^{\infty}G_n$.  Notice that

\begin{center} $D(x, hx) \leq D(x, h_1x) + D(h_1x, hx)$

$=  D(x, h_1x) + D(h_1x, h_0h_1x) =  D(x, h_1x) + D(h_1x, h_1h_0x)$

$= D(x, h_1x) + D(x, h_0x) \leq N + N'$
\end{center}

Then $N + N'$ is a bound on $L(h)$ for $h\in (1_{G_n})_{n=0}^{K-2}\times \prod_{n = K}^{\infty}G_n$ and arguing in this manner back over the finitely many remaining coordinates we obtain a bound on $L(h)$ for all $h\in G$.  We conclude that $G$ is strongly bounded.
\end{proof}

Thus under $(*)$ it is true that a countable product of finite groups is strongly bounded.  This is not the case under the axiom of choice: the product $(\mathbb{Z}/2\mathbb{Z})^{\omega}$ fails to be strongly bounded since it is $\omega$-cofinal.

\end{section}

\end{document}